\documentclass[11pt]{amsart}
\usepackage{graphicx}
\usepackage{a4wide}
\usepackage{amssymb}
\usepackage{amsthm}
\usepackage{amsmath}
\usepackage{amscd} \usepackage{verbatim}
\textheight8.5in \textwidth6.7in \numberwithin{equation}{section}
\usepackage{hyperref}
\usepackage{cleveref}
\theoremstyle{plain}
\newtheorem{theorem}{Theorem}[section]
\newtheorem{corollary}[theorem]{Corollary}
\newtheorem{lemma}[theorem]{Lemma}
\newtheorem{proposition}[theorem]{Proposition}

\theoremstyle{definition}

\newtheorem{example}[theorem]{Example}
\theoremstyle{remark}
\newtheorem*{remark}{Remark}

\newcommand{\SL}{\text {\rm SL}}
\newcommand{\R}{\mathbb{R}}

\newcommand{\Z}{\mathbb{Z}}
\newcommand{\N}{\mathbb{N}}
\newcommand{\C}{\mathbb{C}}

\newcommand{\bbH}{\mathbb{H}}

\newcommand{\htis}{\widehat{T}_{r,N}}
\newcommand{\diff}{\widehat{T}_{r,N}(n) - \widehat{T}_{N-r,N}(n)}

\newcommand{\That}{\widehat{T}}
\newcommand\smod[1]{\ \left(\operatorname{mod} #1\right)}
\newcommand{\abcd}{\left(\begin{smallmatrix} a & b \\ c & d \end{smallmatrix}\right)}
\newcommand{\SLZ}{\SL_2(\Z)}
\newcommand{\calC}{\mathcal{C}}
\renewcommand{\Re}{\operatorname{Re}}
\renewcommand{\Im}{\operatorname{Im}}

\newcommand{\sign}{\operatorname{sgn}}

\begin{document}
\title[The number of parts in residue classes]
{The number of parts in certain residue classes of integer partitions}
\author{Olivia Beckwith and Michael H. Mertens}
\address{Department of Mathematics and Computer Science,
Emory University, Atlanta, Georgia 30322}
\email{obeckwith@gmail.com}
\email{mmerten@emory.edu}

\begin{abstract}
We use the Circle Method to derive asymptotics for functions related to the number of parts of partitions in particular residue classes.
\end{abstract}
\date{May 22nd, 2015}
\maketitle

\section{Introduction and statement of results}\label{sec:intro}

The Circle Method, certainly one of the most important techniques in analytic number theory, originates in the investigation of the partition function. Exploiting the modularity of the generating function of partition function, Hardy and Ramanujan \cite{HR} found their famous asymptotic formula for the number $p(n)$ of partitions of a natural number $n$,
\[p(n)\sim \frac{1}{4n\sqrt{3}}\exp\left(\pi\sqrt{\frac{2n}{3}}\right).\]
Later, Rademacher \cite{Rademacher} was able to refine the method of Hardy-Ramanujan to obtain his exact formula 
\begin{equation}\label{eqRademacher}
p(n)=\frac{2\pi}{(24n-1)^\frac 34} \sum_{k=1}^\infty \frac{A_k(n)}{k}I_\frac 32\left(\frac{\pi\sqrt{24n-1}}{6k}\right),
\end{equation}
where $I_\frac 32$ is the modified Bessel function of the first kind and $A_k(n)$ is a certain Kloosterman sum (see \eqref{eqKloosterman}).

In this paper, we study the number of parts in all partitions in certain congruence classes. The first work addressing the related question of how many parts does a ``generic'' partition of an integer $n$ contain, is the seminal work by Erd\H{o}s and Lehner \cite{ErdosLehner}. To be more precise, they showed that for large $n$, almost all partitions of $n$ contain
\[(1 + o(1)) \frac{\sqrt{6n}}{2 \pi} \log{n}\]
parts.

If $\lambda = (\lambda_0, \dots ,\lambda_k)$ is a partition, i.e. a non-increasing sequence of positive integers, we let 
\begin{equation}
T_{r,N} (\lambda) = | \{ \lambda_j : \lambda_j \equiv r \pmod{N} \} |.
\end{equation}
For a positive integer $n$ we then define 
\begin{equation}
 \htis (n) = \sum_{|\lambda| = n} T_{r,N} (\lambda),
\end{equation} 
where the summation runs over all partitions of size $n$. The quantity $\htis(n)$ counts the number of parts congruent to $r  \pmod{N}$ in all partitions of $n$. For example, all partitions of $5$ are
\[(5),\ (4,1),\ (3,2),\ (3,1,1),\ (2,2,1),\ (2,1,1,1),\ (1,1,1,1,1),\]
hence $\That_{1,3}(5)=13$ and $\That_{2,3}(5)=5$. We will study differences between these functions for $N \ge 3$ and $\gcd(r,N)= 1$. 

A formula giving a lower bound for the number of parts of a partition of $n$ in a residue class $r \pmod{d}$ for a proportion of partitions was proved by Dartyge, Sarkozy, and Szalay  \cite{DaSaSz} for $d < n^{\frac{1}{2} - \epsilon}$. The same authors in \cite{DaSaSz2} proved a formula for the expected number of parts of a partition in a residue class when the parts must be distinct.

In \cite{DaSa} Dartyge and Sarkozy proved an inequality that suggests that some residue classes occur as parts in a partition more frequently than others. They showed that for sufficiently large $n$ and $1 \le r < s \le N$, a positive proportion of the partitions of $n$ satisfy
$$
T_{r,N} (\lambda) - T_{s,N} (\lambda) > \frac{(r + s)\sqrt{n}}{50 rs}.
$$

Despite these results, little is known about the function $\htis(n)$. We will use the Circle Method to prove an asymptotic formula for $\diff$ for $\gcd(r,N) = 1$. 
\begin{theorem}\label{thm:main}
Let $r,N$ be coprime positive integers with $N\geq 3$ and $1\leq r< \tfrac N2$. Then we have that
\begin{align*}
\diff =& \frac{1}{2 \sqrt{2} \varphi (N) N} \left( \sum_{\psi(-1) = -1} \psi( r ')  \sum_{c = 1}^{N - 1} \psi(c) \cot \left( \frac{ \pi c}{N} \right) \right)
 \frac{e^{ \left(\pi \sqrt{ \frac{2}{3} \left( n - \frac{1}{24} \right)} \right)}}{\sqrt{ \left( n - \frac{1}{24} \right)} }\\
 & -\frac{1}{4\sqrt{3}\varphi(N)}\sum_{\psi(-1) = -1} \psi( r ')L(0,\psi)\frac{e^{ \left(\pi \sqrt{ \frac{2}{3} \left( n - \frac{1}{24} \right)} \right)}}{ n - \frac{1}{24}  }+ O\left( n^2 e^{ \left( \frac{\pi}{2} \sqrt{ \frac{2}{3} \left( n - \frac{1}{24} \right)} \right)} \right) ,
\end{align*}
where $\psi$ runs through all odd Dirichlet characters modulo $N$, $L(s,\psi)$ denotes the Dirichlet $L$-series associated to $\psi$, and $r'$ denotes the multiplicative inverse of $r$ modulo $N$.
\end{theorem}
\begin{remark}
In the proof of \Cref{thm:main}, we also give further terms of an asymptotic expansion of $\diff$, see \eqref{eqAsymp}.
\end{remark}

\begin{example}
Let $N=3$ and $r=1$. Then the formula in \Cref{thm:main} simplifies to
$$\That_{1,3} (n) - \That_{2,3} (n) \sim \frac{1}{6 \sqrt{6}} \left(\frac{1} {\sqrt{n - \frac{1}{24}}}-\frac{1} {2\sqrt{2}\left(n - \frac{1}{24}\right)}\right)e^{ \pi \sqrt{ \frac{2}{3} \left( n - \frac{1}{24}\right)} }.$$
We denote by $Q(n)$ the quotient of the left-hand side of the above asymptotic equation by its right-hand side. \Cref{numerics1} contains numerical values of $Q(n)$ for various $n$, illustrating that the above asymptotic is a relatively good approximation for large $n$.
\begin{table}[h!]
\begin{tabular}{|c|c|c|c|c|c|}
\hline 
$n$ & $10$ & $100$ & $1,000$ & $10,000$ & $100,000$  \\ 
\hline
$Q(n)$ & 1.00417 & 1.00142 & 1.00013 & 1.00001 & 1.00000  \\ 
\hline 
\end{tabular}
\vspace{0.5cm}
\caption{Numerics for \Cref{thm:main}}
\label{numerics1}
\end{table}
\end{example}
As it turns out, the generating function of $\diff$ is a weakly holomorphic modular form of weight $\tfrac 12$ for $\Gamma_1(N)$ so that we can essentially follow Rademacher's proof for his exact formula for $p(n)$. 

Next to these differences we also determine asymptotics for $\That_{0,N}$ for arbitrary $N$. In this case, the generating function is \emph{not} modular, but rather the period function of a Maa{\ss}-Eisenstein series times a weakly holomorphic modular form. In this case, we employ a different variant of the Circle Method due to Wright \cite{Wright}, which has essentially been rediscovered and used in various contexts by K. Bringmann, K. Mahlburg, and their collaborators, see e.g. \cite{BM1,BM2}. We use a very convenient formulation due to Ngo and Rhoades \cite{RhoadesNgo}.
\begin{theorem}\label{thm:main2}
For fixed $N\in\N$  we have
\[T_{0,N}(n)=\frac{e^{2\pi\sqrt{\frac n6}}n^{-\frac 12}}{4\pi N\sqrt{2}}\left[\log n-\log\left(\frac{\pi^2}{6}\right)+2\gamma_E-2\log N+O(n^{-\frac 12}\log n)\right],\]
where $\gamma_E=0.577215664...$ is the Euler-Mascheroni constant, as $n\rightarrow\infty$.
\end{theorem}
\begin{example}
We also would like to illustrate the convergence of this asymptotic. Denote by $Q_N(n)$ the quotient of $T_{0,N}(n)$ by the main term in \Cref{thm:main2}. We give some numerical values of $Q_N(n)$ for various $N$ and $n$ in \Cref{numerics2} below.
\begin{table}[h!]
\begin{tabular}{|c|c|c|c|c|c|}
\hline 
$n$ & $10$ & $100$ & $1,000$ & $10,000$ & $100,000$  \\ 
\hline
$Q_1(n)$ & 1.09403 & 1.01393 & 1.00260 & 1.00050 & 1.00029  \\ 
\hline 
$Q_3(n)$ & 1.79224 & 1.06709 & 1.01177 & 1.00247 & 1.00075 \\
\hline
$Q_6(n)$ & -0.81043 & 1.23311 & 1.03137 & 1.00617 & 1.00157 \\
\hline
\end{tabular}
\vspace{0.5cm}
\caption{Numerics for \Cref{thm:main2}}
\label{numerics2}
\end{table}

We observe that the approximation is not quite as good as in \Cref{thm:main} which is due to the larger error term that we have here.
\end{example}
\begin{remark}
Setting $N=1$ in \Cref{thm:main2} recovers Theorem 1.5 in \cite{RhoadesNgo}, where a different combinatorial interpretation for the number $T_{0,1}(n)$ was used.
\end{remark}

In \Cref{secPrelim}, we prove a formula for the generating function $\htis (n)$ and prove a preliminary lemma about odd Dirichlet characters. In \Cref{secEisenstein}, we describe the weight 1 Eisenstein series and relate the generating function for $\diff$ to weight one Eisenstein series. \Cref{secWright} recalls a variant of Wright's Circle Method from \cite{RhoadesNgo} that we use later to prove \Cref{thm:main2}. Finally, in \Cref{secProof1} we prove both our main theorems. 

\section*{Acknowledgements}
The authors would like to thank Ken Ono for suggesting this project and helpful discussions. For suggestions on how to improve an earlier draft of this paper they thank Kathrin Bringmann, Frank Garvan, Karl Mahlburg, Sharon Garthwaite, Jos\'e-Miguel Zapata-Rol\'on, and Sander Zwegers.

\section{Preliminary Results}\label{secPrelim}
We prove a formula for the generating function for $\htis (n)$.
\begin{lemma}\label{thm:qseries}
$\htis$ has the following generating function,
$$ \sum_{n=1}^{\infty} \htis (n) q^n = \left(\prod_{n \ge 1}\frac{1}{ (1 - q^n)} \right) \left( \sum_{n=1}^{\infty} \left( \sum_{\substack{ d |n \\ d \equiv r \smod{N}} } q^n \right) \right). $$
\end{lemma}
\begin{proof} 
Note that 
$$ \frac{q^m}{(1 - q^m)^2} = \sum_{k = 1}^{\infty} k q^{km}.$$

Then, modifying the proof of Euler's formula for the generating function of $p(n)$, we have that

$$
\sum_{n =  0}^{\infty} a(n)q^n = \frac{q^m}{(1 - q^m)^2} \cdot \prod_{\substack{n \geq 1 \\ n \neq m}} \frac{1}{1 - q^n},
$$
where $a(n)$ equals the number of times $m$ appears as a part in a partition of $n$. 

Thus, summing over $m \equiv r \pmod{N}$, we have
{\allowdisplaybreaks 
\begin{align*} 
\sum_{n=1}^{\infty} \htis (n) q^n &=\sum_{\substack{m\geq 1 \\ m \equiv r \smod{N}}} \frac{q^m}{(1 - q^m)^2} \cdot \prod_{ \substack{n\geq 1\\ n \neq m}} \frac{1}{1 - q^n}\\
&= \left( \prod_{n\ge 1} \frac{1}{1 - q^n} \right) \sum_{\substack{m\ge 1 \\ m \equiv r \smod{N}}} \left( \sum_{k =1}^{\infty} q^{km} \right) \\
&= \left( \prod_{n \ge 1} \frac{1}{(1 - q^n)} \right) \left( \sum_{n=1}^{\infty} \left( \sum_{\substack{ d |n \\ d \equiv r \smod{N}} } q^n \right) \right)  . 
\end{align*}
}
This proves the lemma.
\end{proof}

The next corollary, which shows that the odd Dirichlet characters obey the same orthogonality relations as the usual group of Dirichlet characters, will be useful to us in relating the generating function of $\diff$ to a combination of weight $1$ Eisenstein series. First, we will show a slightly more general lemma about characters of finite abelian groups.

\begin{lemma}\label{lemChar}
Let $G$ be a finite abelian group containing an element $u\in G\setminus\{1\}$ with $u^2=1$. Fix such a $u$, then we have the following equality,
\[\frac{2}{n}\sum\limits_{\substack{\psi\in \widehat{G} \\ \psi(u)=-1}}\psi(g)=\begin{cases} 1 & g=1 \\ -1 & g=u \\ 0 & \text{otherwise}.\end{cases},\]
where $n:=|G|$ and $\widehat{G}:=\operatorname{Hom}(G,\C^*)$ denotes the dual group of $G$.
\end{lemma}
\begin{proof}
Since $G$ is finite and abelian, it is a direct product of cyclic groups 
\[G\cong C_{d_1}\times ...\times C_{d_\ell}\]
with $d_1|...|d_\ell$. Since we can define characters on each component separably, we can assume without loss of generality that $G=\langle g\rangle$ is cyclic. Note that the existence of an element $u$ of order $2$ in $G$ assures that at least one of the cyclic factors above has even order. Now if $\zeta_n$ is a primitive $n$th root of unity, each character $\psi$ of $G$ is uniquely determined by setting
\[\psi(g)=\zeta_n^k\]
for some $k\in\{0,...,n-1\}$. Since $u=g^\frac n2$, we see that the condition $\psi(u)=-1$ forces $k$ to be odd. Thus we have for any $j\in\{0,...,n-1\}$ that
\begin{align*}
&\sum\limits_{\substack{\psi\in \widehat{G} \\ \psi(u)=-1}}\psi(g^j)
=\sum\limits_{\substack{k=0 \\ k\text{ odd}}}^n\zeta_n^{jk}
=\zeta^j\sum\limits_{k=0}^{\frac n2}\zeta_n^{2jk}
=\zeta^j\sum\limits_{k=0}^{\frac n2}\zeta_{\frac n2}^{jk}.
\end{align*}
Now, unless $j$ is a multiple of $\tfrac n2$, the last sum vanishes because it ranges over all $d$th roots of unity, where $d=\gcd(jk,\tfrac n2)$. For $j=0$, the whole expression obviously becomes $\tfrac n2$, for $j=\tfrac n2$, we obtain $-\tfrac n2$, proving the assertion.
\end{proof}

\begin{corollary}\label{thm:dirchars}
Let $\gcd(r,N) =1$. Then

$$ \phi_{r,N} (n) := \frac{2}{\varphi(N)} \sum_{\substack{ \psi \smod{N} \\ \psi(-1) = -1}} \psi( n \cdot r') = 
\begin{cases}
1, & \text{if }n \equiv r \pmod{N} \\
-1, & \text{if }n \equiv -r \pmod{N} \\
0 & \text{otherwise}
\end{cases} $$

where $\varphi(N):=|(\Z/N\Z)^*|$ denotes Euler's totient function. The summation over $\psi$ runs over all odd characters modulo $N$ and $r'$ is any integer such that $rr' \equiv 1 \pmod{N}$. 
\end{corollary}
\begin{proof}
This is an immediate consequence of \Cref{lemChar}.
\end{proof}

\section{Weight 1 Eisenstein Series}\label{secEisenstein}
In this section, we recall the relevant facts about Eisenstein series of weight $1$. For a general reference, the reader is referred to Section 4.8. in \cite{DS}. We use the usual convention that $\tau$ is a variable from the complex upper half-plane and $q=e^{2\pi i\tau}$. For a function $f:\bbH\rightarrow\C$, a weight $k\in\Z$, and a matrix $\gamma=\abcd\in\SLZ$ we define the operator
\[f[\gamma]_k(\tau):=(c\tau+d)^{-k}f\left(\frac{a\tau+b}{c\tau+d}\right).\]

For $v = (c_v, d_v) \in ((\Z / N \Z)^*)^2$ with order $N$ we define 
$$
g_1^{v} (\tau) = \delta (c_v) \zeta^{d_v} (1) + \frac{2 \pi i}{N} \left( \frac{\overline{c_v}}{N} - \frac{1}{2} \right) - \frac{2 \pi i}{N} \sum_{n=1}^{\infty} \left( \sum_{\substack{ m|n \\ \frac{n}{m} \equiv c_v \smod{N}}} \sign(m) e^{\frac{2 \pi i d_v m}{N}} \right) q^{\frac{n}{N}}.
$$

In this formula, $\overline{ c_v}$ denotes the integer such that $0 \le \overline{c_v} < N$ and $c_v \equiv \overline{c_v} \pmod{N}$. We also define
$$
\delta( c_v ) := \begin{cases}
       1 & \overline{c_v} = 0, \\
       0 & \text{ otherwise,}
     \end{cases}   
$$
and the function $\zeta^d(k)$ by
$$
\zeta^{d} (s) := \sum_{\substack{n\in\Z \\ n\equiv d\smod N}}\frac{
1}{n^s}
$$
for $\Re(s)>1$ and otherwise by analytic continuation. We shall mainly need the special value (see \cite{DS}, Equation (4.22) and Exercise 4.4.5) 
\begin{equation}\label{eqZeta}
\zeta^d(1)=\frac{\pi i}{N} + \frac{\pi}{N} \cot \left( \frac{\pi d}{N} \right)\qquad (\gcd(d,N)=1).
\end{equation}

The formula for $g_1^{v} (\tau)$ is very similar to the typical Eisenstein series for $k \ge 3$, but contains a correction term. One can show that $g_1^v(\tau)$ is a weight 1 modular form with respect to the principal congruence subgroup $\Gamma(N)$, and satisfies the equation
$$
g_1^v [\gamma]_1 (\tau) = g_1^{\gamma(v)} (\tau),
$$
with
$$
\gamma(v) = (a c_v + c d_v, b c_v + d d_v    ).
$$
for any $\gamma=\abcd\in \SLZ$.

To obtain forms which are weight 1 invariant with respect to $\Gamma_1(N)$, one generally takes special linear combinations of the $g_1^v$ functions as follows: Let $\psi, \chi$ be Dirichlet characters modulo $u$ and $v$ respectively with $uv = N$ such that $\chi$ is primitive and $\psi \chi$ is odd. Define $G_1^{\psi, \chi}$ as follows:
\begin{equation}\label{equation:G1}
G_1^{\psi, \chi} (\tau) = \sum_{c = 0}^{u-1} \sum_{d = 0}^{v-1} \sum_{e = 0}^{u-1} \psi(c) \overline{\chi}(d) g_1^{(cv, d + ev)} (\tau).
\end{equation}
In our arguments, we choose to take $v$ to be $1$ so that $\chi$ is trivial and $\psi$ is an odd character with respect to modulus $N$. We let $E_1^{\psi} (\tau)$ denote the normalized series given by $- \frac{1}{2 \pi i } G_1^{\psi, 1} (\tau)$. The Fourier series of $E_1^{\psi} (\tau)$ is given by (see \cite{DS}, p. 140)
\begin{equation}\label{equation:Epsi}
E_1^{\psi} (\tau) = L(0, \psi) + 2 \sum_{n =1}^{\infty} \left(\sum_{d |n} \psi(d)\right) q^n.
\end{equation}

The next result connects the $E_1^{\psi}$ series to the generating function for $\diff$. 
\begin{proposition}\label{thm:eisenstein}
If $N \ge 3$ and $\gcd(r,N) = 1$, then 
$$
G_{r,N} (q) := \sum_{n \ge 1} \left( \diff \right) q^n= \frac{1}{\varphi(N)} \frac{q^{\frac{1}{24}}}{\eta(\tau)} \left(c_{r,N} + \sum_{\substack{\psi \smod{N} \\ \psi(-1) = -1}} \psi(r') E_1^{\psi}(\tau)\right),
$$

where 
$$c_{r,N} = - \sum_{\substack{\psi \smod{N} \\ \psi(-1) = -1}} \psi(r')L(0,\psi)$$ 
and $L(s,\psi)$ denotes the Dirichlet $L$-series associated to $\psi$.
\end{proposition}
\begin{proof}
First, we rewrite $G_{r,N} (q)$ using Lemma \ref{thm:qseries}.

\begin{align*}
G_{r,N} (q) &=\prod_{n=1}^\infty \frac{1}{1 - q^n} \left( \sum_{m=1}^\infty \left(\sum_{\substack{d|m \\ d\equiv r\smod{N}}}-\sum_{\substack{d|m \\ d\equiv -r\smod{N}}}\right)q^m\right).
\end{align*}

By \Cref{thm:dirchars} we see that the coefficient of $q^m$ is $\sum_{d|m} \phi_{r,N} (d)$, so that we can write
$$
G_{r,N} (q)  = \frac{1}{\varphi(N)}\prod_{n=1}^\infty \frac{1}{1 - q^n} \left( \sum_{n = 1}^\infty \left( \sum_{d | n}  \sum_{\substack{ \psi \smod{N} \\ \psi(-1) = -1}} \psi(d) \psi(r') \right) q^n \right) .
$$
Finally, we write this expression in terms of the Eisenstein series given in \eqref{equation:Epsi}. 
\begin{align*}
G_{r,N}(q) &= \frac{1}{\varphi(N)} \frac{q^{\frac{1}{24}}}{\eta(\tau)} \left( \sum_{\substack{ \psi \smod{N} \\ \psi(-1) = -1}} \left( \psi(r') E_{1}^{\psi} (\tau) - \psi(r') L(0,\psi)  \right) \right) \\
&= \frac{1}{\varphi(N)} \frac{q^{\frac{1}{24}}}{\eta(\tau)} \left(c_{r,N}+ \sum_{\substack{ \psi \smod{N} \\ \psi(-1) = -1}} \psi(r') E_1^{\psi} (\tau) \right)  .\\
\end{align*}
\end{proof}

\bigskip

Our proof will also require some information about the behavior of the Eisenstein series in \Cref{thm:eisenstein} near the cusps, that is, near $\tau = \frac{h}{k}$. 

Let $E_{r,N} (\tau)$ denote the Eisenstein series in \Cref{thm:eisenstein}, that is 
$$E_{r,N}(\tau) := \sum_{\substack{\psi \smod{N} \\ \psi(-1) = -1}} \psi(r') E_1^{\psi}(\tau).$$
\begin{lemma}\label{thm:cuspbehavior}
Let $r,N$ be fixed positive integers with $\gcd(r,N) = 1$. Let $h,k$ be integers with $h \le k$ and $\gcd(h,k) = 1$. 
Let $H$ be such that $1 \le H \le k$ and $h H \equiv -1 \pmod{k}$. Let 
$$
\alpha_{h,k} := 
\begin{pmatrix}
-h & \frac{hH+1}{k} \\
-k & H
\end{pmatrix}
$$

Then
$$
E_{r,N} [ \alpha ] _1 (\tau) = \sum_{n = 0}^{\infty} a_n(h,k) q^{\frac{n}{N}} 
$$
where the following are true:
\begin{equation}\label{equation:cuspval}
a_0 (h,k) = \sum_{\substack{ \psi \smod{N} \\ \psi(-1) = -1}} c_{\psi} (h,k) \psi (r ')
\end{equation} 
where
\begin{equation}
c_{\psi}(h,k)=  -\frac{1}{ 2 \pi i} \sum_{c=0}^{N-1} \sum_{e = 0}^{N-1} \psi (c)  \left(\delta( -hc -ke) \zeta^{( \frac{hH +1}{k})c +H e}(1) + \frac{2 \pi i}{N} \left( \frac{\overline{(-hc - ke)}}{N} - \frac{1}{2} \right)\right)
\end{equation}
and $rr' \equiv 1 \pmod{N}.$

Also, we have the following:
$$|a_0(h,k)| \le C_0, $$
and for all $n \ge 1$, 
$$
|a_n (h,k)| \le C_1 n,
$$
where $C_0, C_1,$ depend on $N$ but do not depend on $h$ or $k$. 
\end{lemma}

\begin{proof}
Recall that
$$
E_{r,N} [ \alpha_{h,k}]_1 (\tau) = \sum_{\substack{\psi \smod{N} \\ \psi(-1) = -1}} \psi(r') E_1^{\psi}[ \alpha_{h,k} ]_1(\tau).
$$
We have
$$
E_1^{\psi} [ \alpha_{h,k} ]_1 (\tau) = -\frac{1}{ 2 \pi i } \sum_{c = 0}^{N-1} \sum_{e = 0}^{N-1} \psi(c)  g_1^{\alpha_{h,k}(c,e)} (\tau)
$$
where
$$
\alpha_{h,k}(c,e) = \left( -h c -k e, \left( \frac{hH +1}{k} \right)c +H e \right).
$$

Using the Fourier expansion for $g_1^{v} (\tau)$, we obtain the following expressions for the coefficients of $E_1^{\psi} [ \alpha ] (\tau)$:
$$
a_0 (h,k)= \sum_{\substack{ \psi \smod{N} \\ \psi(-1) = -1}} \psi(r') c_{ \psi}(h,k)
$$
where
$$
c_{\psi} (h,k)= -\frac{1}{2 \pi i} \sum_{c=0}^{N-1} \sum_{e = 0}^{N-1} \psi (c)  \left(\delta( -hc -ke) \zeta^{( \frac{hH +1}{k})c +H e}(1) + \frac{2 \pi i}{N} \left( \frac{\overline{(-hc - ke)}}{N} - \frac{1}{2} \right)\right).
$$
and for $n \ge 1$:
$$
a_n (h,k)= \sum_{\substack{\psi \smod{N} \\ \psi ( -1) = -1}} \psi (r') \left( \sum_{c=0}^{N-1} \sum_{e = 0}^{N-1} \left( \sum_{ \substack{d | n \\ \frac{n}{d} \equiv -hc - ke \smod{N}}} \sign(d) e^{\frac{2 \pi i }{N} d (  \frac{hH + 1}{k} c + He)} \right) \right).
$$
The inner sum is bounded by the divisor counting function $2d(n)$ which is less than $2n$, so the whole expression is bounded by $\varphi (N) N^2 n$. Note that much stronger bounds for $d(n)$ exist, in fact $d(n) = O( n^{\epsilon})$ for any $\epsilon > 0$, however we require only the crude bound $d(n) \le n$.  This proves the second inequality in the proposition.

\end{proof}

We can describe the behavior of the eta function near the cusps using the transformation formula for $\eta(\tau)$. For $\tau' = \frac{H}{k} + \frac{i}{z}$, and $\tau = \frac{h}{k} + \frac{i z}{k^2}$, we have $\alpha_{h,k}^{-1} ( \tau ) = \tau'$. 

\begin{theorem}\label{thm:etatransf}
Let $i,s, h,k, H, \tau, \tau'$ be as described above. Then we have
$$
\frac{q^{\frac{1}{24}}}{\eta ( \tau)} = \frac{e^{2 \pi i \frac{\tau '}{24}}}{\eta(\tau')} \left( \frac{z}{k} \right)^{\frac{1}{2}} e^{- \frac{ \pi z}{12 k^2} + \frac{\pi}{12 z} + \pi i s(-H,k)},
$$
where $s(h,k)$ denotes the usual Dedekind sum,
\begin{equation}\label{eqDedekind}
s(h,k) = \sum_{r = 1}^{k-1} \frac{r}{k} \left( \frac{hr}{k} - \Big\lfloor \frac{hr}{k} \Big\rfloor - \frac{1}{2} \right). 
\end{equation}
\end{theorem}

We refer the reader to Theorem 5.1 in \cite{Apostol} for a proof of this fact. 

\section{Asymptotic expansions \`a la Wright}\label{secWright}
In \cite{RhoadesNgo}, Ngo and Rhoades established very general Propositions to widen the applicability of Wright's Circle Method. For the convenience of the reader, we recall their results here.

Let $\xi(q)$ and $L(q)$ be analytic functions for $|q|<1$ and $q\notin\R_{\leq 0}$, such that $$\xi(q)L(q)=:\sum_n a(n)q^n$$ is analytic for $|q|<1$. Further assume the following hypotheses, where $0<\delta<\tfrac\pi 2$ and $c>0$ are fixed constants.
\begin{enumerate}
\item\label{hypo1} As $\sigma\rightarrow 0$ in the cone $|\arg \sigma|<\tfrac \pi 2-\delta$ and $|\Im \sigma|\leq\pi$ we have either
\begin{equation}\label{LpolyMajor}
L(e^{-\sigma})=\sigma^{-B}\left(\sum\limits_{\ell=0}^{k-1}\alpha_\ell \sigma^\ell+O_\delta(\sigma^k)\right)
\end{equation}
or
\begin{equation}\label{LlogMajor}
L(e^{-\sigma})=\frac{\log \sigma}{\sigma^B}\left(\sum\limits_{\ell=0}^{k-1}\alpha_\ell \sigma^\ell+O_\delta(\sigma^k)\right)
\end{equation}
for some $B\in\R$. 
\item\label{hypo2} As $\sigma\rightarrow 0$ in the cone $|\arg \sigma|<\tfrac \pi 2-\delta$ and $|\Im \sigma|\leq \pi$ we have
\begin{equation}\label{xiMajor}
\xi(e^{-\sigma})=\sigma^\beta e^{\frac{c^2}{\sigma}}\left(1+O_\delta(e^{-\frac \gamma \sigma})\right)
\end{equation}
for real constants $\beta\geq 0$ and $\gamma>c^2$.
\item\label{hypo3} As $\sigma\rightarrow 0$ in the cone $\tfrac \pi 2-\delta\leq |\arg \sigma|<\tfrac \pi 2$ and $|\Im \sigma|\leq \pi$ one has
\begin{equation}\label{Lminor}
|L(e^{-\sigma})|\ll_\delta |\sigma|^{-C},
\end{equation}
where $C=C(\delta)>0$.
\item\label{hypo4} As $\sigma\rightarrow 0$ in the cone $\tfrac \pi 2-\delta\leq |\arg \sigma|<\tfrac \pi 2$ and $|\Im \sigma|\leq \pi$ one has
\begin{equation}\label{ximinor}
|\xi(e^{-\sigma})|\ll_\delta\xi(|e^{-\sigma}|)e^{-K\Re\left(\frac 1{\sigma}\right)},
\end{equation}
where $K=K(\delta)>0$.
\end{enumerate} 
In the case of \eqref{LpolyMajor} we say that $L$ is of \emph{polynomial type near $1$}, in that of \eqref{LlogMajor} we call $L$ of \emph{logarithmic type near $1$}.

In the language of Wright's Circle Method, the hypotheses in \eqref{LpolyMajor}--\eqref{xiMajor} determine the asymptotics of $L$ and $\xi$ on the \emph{major arc}, and those in \eqref{Lminor} and \eqref{ximinor} that on the \emph{minor arc}.
We require the following two propositions (Propositions 1.8 and 1.10 in \cite{RhoadesNgo}) for our results, the first of which is originally due to Wright \cite{Wright}.
\begin{proposition}\label{Wrightpoly}
Suppose the hypotheses (1)--(4) are satisfied and that $L$ has polynomial type near $1$. Then there is an asymptotic expansion
\begin{equation}\label{asympoly}
a(n)=e^{2c\sqrt{n}}n^{\frac 14(2B-2\beta-3)}\left(\sum\limits_{r=0}^{M-1}p_rn^{-\frac r2}+O(n^{-\frac M2})\right),
\end{equation}
where
\begin{equation}
p_r=\sum\limits_{s=0}^r \alpha_sw_{s,r-s}
\end{equation}
with $\alpha_s$ as in \eqref{LpolyMajor} and 
\begin{equation}
w_{s,r}=\frac{c^{s+\beta-B+\frac 12}}{(-4c)^r2\pi^\frac 12}\cdot\frac{\Gamma\left(s+\beta-B+r+\frac 32\right)}{r!\Gamma\left(s+\beta-B-r+\frac 32\right)}
\end{equation}
for the coefficients $a(n)$ of $\xi(q)L(q)$ as $n\rightarrow\infty$.
\end{proposition}
\begin{proposition}\label{Wrightlog}
Suppose hypotheses (1)--(4) are satisfied and that $L$ has logarithmic type near $1$ such that $B-\beta=\tfrac 12$, with $B$ and $\beta$ as in equations \eqref{LlogMajor} and \eqref{xiMajor} respectively. Then we have 
\[a(n)=-e^{2c\sqrt{n}}n^{-\frac 12}\frac{\alpha_0}{4\pi^\frac 12}\left(\log n-2\log c+O(n^{-\frac 12}\log n)\right)\]
as $n\rightarrow\infty$.
\end{proposition}
\section{Proof of the main theorems}\label{secProof1}
\subsection{Proof of Theorem \ref{thm:main}}
We will follow the proof of Rademacher's formula for the partition function, as described by Apostol in chapter 5 of \cite{Apostol}. Throughout the proof, let $r$ and $N$ be fixed coprime  integers with $N \ge 3$.

By Cauchy's integral formula, we have:
$$
\diff =  \frac{1}{2 \pi i} \int_\calC \frac{G_{r,N} (q)}{q^{n+1}} dq, 
$$
where $\calC$ is any positively oriented contour lying inside the unit circle, which contains the origin in its interior. By \Cref{thm:eisenstein}, we can decompose this integral into two integrals, as follows,
\begin{equation}\label{equation:diff}
\diff = T_1 + T_2 := \frac{1}{2 \pi i \varphi(N)} c_{r,N} \int_\calC   \frac{q^{\frac{1}{24}}}{\eta(\tau)} \frac{1}{q^{n+1}} dq + \frac{1}{2 \pi i \varphi(N)} \int_\calC  \frac{ q^{\frac{1}{24}} E_{r,N} (\tau) }{\eta(\tau) q^{n+1}} dq.
\end{equation}

The first integral is basically the one Rademacher considered, which yields (see \eqref{eqRademacher})
$$
T_1 =\frac{c_{r,N}}{ \varphi(N)} \frac{2 \pi}{(24)^{\frac{3}{4}}} \left( n - \frac{1}{24} \right)^{- \frac{3}{4}} \sum_{k = 1}^{\infty} A_k(n) k^{-1} I_{\frac{3}{2}} \left( \frac{\pi}{k} \sqrt{ \frac{2}{3} \left( n - \frac{1}{24} \right) } \right), 
$$
where $I_{\frac{3}{2}}$ is the order $\tfrac{3}{2}$ modified Bessel function of the first kind,
\begin{equation}\label{eqKloosterman}
A_k(n) = \sum_{\substack{ 0 \le h < k \\ (h,k) =1 }} e^{ \pi i s(h,k) - 2 \pi i n \frac{h}{k}},
\end{equation}
and $s(h,k)$ is as in \eqref{eqDedekind}.

For the rest of the subsection, we will use Rademacher's technique to obtain an asymptotic formula for $T_2$. We define the contour $\calC$ to be given by $q = e^{2 \pi i \tau}$ where $\tau$ follows Rademacher's path of integration $P(n)$ which takes $\tau$ from $i$ to $i+1$ by going along the upper arcs of the Farey circles $C_{h,k}$ where $\frac{h}{k}$ is in the Farey sequence of order $n$. That is, $\gcd(h,k) =1 $, and $1 \le h \le k \le n$. Let $\gamma(h,k)$ denote the upper arc of the Ford circle $C_{h,k}$. Changing variables from $q$ to $\tau$, we have:
$$
T_2 = \frac{1}{\varphi (N)} \sum_{\substack{ 0 \le h \le k \le n \\ \gcd(h,k) = 1}} \int_{\gamma(h,k)}   \frac{ q^{\frac{1}{24}} E_{r,N} ( \tau )}{\eta(\tau) q^n} d\tau.
$$

As in \cite{Apostol}, we make the change of variables $z = -ik^2 \left( \tau - \frac{h}{k} \right)$, so that $\tau = \frac{h}{k} + \frac{iz}{k^2}$. This maps $\gamma (h,k)$ to an arc of the circle of radius $\frac{1}{2}$ centered at $\frac{1}{2}$. The contour goes in the clockwise direction from  the image of the left end point of $\gamma_(h,k)$ to the image of the right endpoint of $\gamma(h,k)$. It is well-known (see Theorem 5.8 in \cite{Apostol}) that if $\frac{h_1}{k_1} < \frac{h}{k} < \frac{h_2}{k_2}$ are adjacent in the Farey sequence of order $M$, then the image under the change of variables from $\tau$ to $z$ of the point where $\gamma(h_1, k_1)$ and $\gamma(h,k)$ intersect is given by 
$$z_1(h,k) =  \frac{k^2}{k^2 + k_1^2} + i \frac{k k_1}{k^2 + k_1^2}.$$ 
Similarly, the point where $\gamma(h,k)$ meets $\gamma (h_2, k_2)$ is mapped to 
$$z_2(h,k) = \frac{k^2}{k^2 + {k_2}^2} - \frac{i k k_2}{k^2 + {k_2}^2}.$$ 
This gives us the following: 
$$
T_2=\frac{i}{ \varphi (N)} \sum_{h,k} \frac{1}{k^2} \int_{z_1(h,k)}^{z_2(h,k)}   \frac{ q^{\frac{1}{24}} E_{r,N} (\tau) }{\eta(\tau) q^n} dz.
$$


Now we will estimate the integrand by its behavior near the cusps. Let $a_n(h,k)$ be as in \Cref{thm:cuspbehavior}. First we decompose the integrand,

\begin{align*}
 &\frac{\left( q^{\frac{1}{24}} E_{r,N} (\tau) \right)}{\eta(\tau) q^{n}} = \Psi_1 (\tau) + \Psi_2 (\tau) :=\\
 & -i \left( \frac{z}{k} \right)^{- \frac{1}{2}} \frac{1}{q^n} e^{- \frac{ \pi z}{12 k^2} + \frac{\pi}{12 z} + \pi i s(-H,k)} a_0(h,k)+ \left( \frac{\left( q^{\frac{1}{24}} E_{r,N} (\tau) \right)}{\eta(\tau) q^{n}} + i \left( \frac{z}{k} \right)^{- \frac{1}{2}} \frac{1}{q^n} e^{- \frac{ \pi z}{12 k^2} + \frac{\pi}{12 z} + \pi i s(-H,k)} a_0(h,k) \right).
\end{align*}
Now we show that $\Psi_2 ( \tau)$ is negligible. We can adjust to the contour of integration so that we are integrating along the chord adjoining $z_1 (h,k)$ and $z_2 (h,k)$ instead of the arc. This yields
\begin{align*}
| \Psi_2 (\tau) | &= \left| \left( \frac{z}{k} \right)^{- \frac{1}{2}}  e^{- \frac{\pi z}{12 k^2} + \frac{\pi}{12 z} +  \pi i s(h,k)}  \frac{1}{q^n} \left( \sum_{m=1}^{\infty} ( \sum_{Nj + \ell =m} a_{\ell}(h,k) p(j) ) e^{2 \pi i \frac{m}{N}( \frac{H}{k} +  \frac{i}{z})} \right)\right| \\
&\le \frac{\sqrt{k}}{\sqrt{ |z|}} e^{2 \pi \Re(z) n /k^2 } \left( \sum_{m=1}^{\infty} (\sum_{Nj+\ell =m}^{\infty} a_{\ell}(h,k) p(j) ) e^{  - 2 \pi \Re\left(\frac 1z\right) (\frac{m}{N} - \frac{1}{24})}  \right).
\end{align*}

For $z$ inside the circle, $\Re\left( \tfrac{1}{z} \right) \ge 1$ (\cite{Apostol}, page 107). By Lemma \ref{thm:cuspbehavior}, we have $|a_n (h,k)| \le C_1 n$ for $n \ge 1$. Thus:
$$
\sum_{m = 1}^{\infty} \left( \sum_{i + j =m}^{\infty} a_i(h,k) p(j) \right) e^{  - 2 \pi \Re\left( \frac 1z\right) \left(m - \frac{1}{24}\right)}  \le \sum_{m=1}^{\infty}  C_1 m^2 p(m) e^{- 2 \pi ( m - \frac{1}{24})}  
$$
The known asymptotics for $p(n)$ ensure that the sum is convergent.

By Theorem 5.9 in \cite{Apostol}, we have $\Re(z) \le |z| \le \frac{\sqrt{2}k}{n}$ on the chord, thus we have 
$
e^{2 \pi \Re(z) n /k^2} \le e^{8 \pi}.
$

Using the formulas for $z_1(h,k), z_2(h,k)$, we have that for $z$ on the chord,$$
|z| \ge \text{min} \{ \text{Re} (z_1(h,k)), \text{Re} (z_2(h,k)) \} \ge \frac{1}{2 n^2}.
$$

Thus we have the following:
$$
| \Psi_2 (\tau) | \le C k^{\frac{1}{2}} n,
$$
for some constant $C$, independent of $h,k,$ or $n$.

The length of the chord at most $\frac{2 \sqrt{2} k}{n}$, so we have
$$
\int_{z_1 (h,k)}^{z_2(h,k)} \Psi_2 (\tau) dz \le C 2 \sqrt{2} k^{\frac{3}{2}}. 
$$ 

Considering all the integrals over $\Psi_2$ that contribute the calculation of $T_2$, we estimate:
$$
\left|\frac{i}{\varphi(N)} \sum_{h,k} \frac{1}{k^2}  \int_{z_1 (h,k)}^{z_2(h,k)} \Psi_2 (\tau) dz\right| \le \frac{C 2 \sqrt{2}}{\varphi (N)} n^{\frac{3}{2}}.
$$

This allows us to approximate $T_2$ as follows:
\begin{align*}
T_2 &= \frac{i}{\varphi(N)} \sum_{h,k} \frac{1}{k^2} \int_{z_1 (h,k)}^{z_2 (h,k)}  -i \left( \frac{z}{k} \right)^{- \frac{1}{2}} \frac{1}{q^n} e^{- \frac{ \pi z}{12 k^2} + \frac{\pi}{12 z} + \pi i s(-H,k)} a_0(h,k)  dz + \frac{i}{\varphi(N)} \sum_{h,k} \frac{1}{k^2}  \int_{z_1 (h,k)}^{z_2(h,k)} \Psi_2 (\tau) dz \\
&= \frac{1}{\varphi(N)} \sum_{h,k} \frac{1}{k^2} \int_{z_1 (h,k)}^{z_2 (h,k)}  \left( \frac{z}{k} \right)^{- \frac{1}{2}} \frac{1}{q^n} e^{- \frac{ \pi z}{12 k^2} + \frac{\pi}{12 z} + \pi i s(-H,k)} a_0(h,k)  dz + O(n^{\frac{3}{2}}).
\end{align*}

To evaluate the integral, we adjust the contour of integration.Let $x_1 (n)$ be the point on the upper half of the circle $K$ with $|x_1 (n)| = \frac{1}{n}$, and let $x_2(n)$ be the point on the lower half of the circle with $|x_2(n)| = \frac{1}{n}$, and we rewrite the integral as $\int_{z_1 (h,k)}^{z_2(h,k)} = \int_{x_1(n)}^{x_2(n)} - \int_{x_1 (n)}^{z_1(h,k)} - \int_{z_2 (h,k)}^{x_2(n)}$. 

We will show that the second two integrals are negligible. On the arc between $x_1(n)$ and $z_1(h,k)$, and on the arc between $z_2(h,k)$ and $x_2(n)$, we have $\frac{1}{n} \le |z| \le \frac{\sqrt{2}k}{n}$. For $z$ on the circle, $\Re\left( \tfrac{1}{z}\right) = 1$. Combining these facts, we have that the integrand is bounded as follows:
\begin{align*}
\left| \frac{1}{k^2} \left( \frac{k}{z} \right)^{\frac{1}{2}} a_0(h,k) e^{-\frac{2 \pi i n h}{k}} e^{\frac{2 \pi n z}{k^2}}
e^{- \frac{ \pi z}{12 k^2} + \frac{\pi}{12 z} + \pi i s(h,k)} \right| 
&\le 
\frac{1}{k^2} \sqrt{k} |z|^{- \frac{1}{2}} |a_0(h,k)| e^{2 \pi n \Re(z) / k^2 - \frac{ \pi \Re(z)}{12 k^2} + \frac{ \pi \Re\left(\frac 1z\right)}{12}} \\
& \le k^{- \frac{3}{2}} \sqrt{n} C_0 k e^{2 \pi \sqrt{2}/k + \frac{ \pi }{12}}  \\
&\le C_0  e^{2 \sqrt{2} \pi + \frac{\pi}{12}} k^{- \frac{1}{2}}  n^{\frac{1}{2}}
\end{align*} 
Thus, making use of the fact that the length of both arcs is bounded by $\pi$, we have the following bound for the integral over these two arcs:
$$
\left| \left(\int_{x_1(n)}^{z_1 (h,k)} + \int_{z_2 ( h,k)}^{x_2(n)}\right) \left(   \frac{1}{k^2} \left( \frac{k}{z} \right)^{\frac{1}{2}} a_0(h,k) e^{-\frac{2 \pi i n h}{k}} e^{\frac{2 \pi n z}{k^2}}
e^{- \frac{ \pi z}{12 k^2} + \frac{\pi}{12 z} + \pi i s(h,k)} \right) \right|\le C_0 \pi e^{2 \sqrt{2} \pi + \frac{\pi}{12}}   .
$$  

Summing over $h,k$, we get the total contribution to the asymptotic from these arcs,
$$
\left| \frac{1}{\varphi(N)} \sum_{h,k} \left(\int_{x_1(n)}^{z_1 (h,k)} + \int_{z_2 ( h,k)}^{x_2(n)} \right) \left(   \frac{1}{k^2} \left( \frac{k}{z} \right)^{\frac{1}{2}} a_0(h,k) e^{-\frac{2 \pi i n h}{k}} e^{\frac{2 \pi n z}{k^2}}
e^{- \frac{ \pi z}{12 k^2} + \frac{\pi}{12 z} + \pi i s(h,k)} \right) \right| \le C_0 \pi e^{2 \sqrt{2} \pi + \frac{\pi}{12}} n^2.
$$

Therefore, we can rewrite $T_2$ as follows:
$$
T_2 =  \frac{1}{\varphi(N)} \sum_{h,k} B_k(n) \int_{x_1(n)}^{x_2(n)} \frac{1}{k^{\frac{3}{2}}} \left( \frac{1}{z} \right)^{\frac{1}{2}}  e^{\frac{2 \pi n z}{k^2}- \frac{ \pi z}{12 k^2} + \frac{\pi}{12 z}}
dz + O(n^2),
$$
where

$$
B_k(n) = \sum_{ \substack{ 1 \le h \le k \\ \gcd(h,k) = 1}} a_0(h,k) e^{\pi i s(h,k) - \frac{2 \pi i n h}{k}}.
$$
We change variables by taking $t = \frac{\pi}{12 z}$. This yields the following:
$$
T_2 =  -\left(\frac{\pi}{12}\right)^{\frac{1}{2}} \frac{1}{\varphi(N)} \sum_{h,k} B_k(n) k^{- \frac{3}{2}} \int_{\frac{\pi}{12} - i \frac{\pi}{12} \sqrt{n^2 - 1}}^{\frac{\pi}{12} + i \frac{\pi}{12} \sqrt{n^2 -1}}  t^{-\frac{3}{2}} 
e^{\frac{2 \pi^2 n}{12 t k^2}- \frac{ \pi^2 }{t 12^2 k^2} + t}  dt + O(n^2).
$$

Finally, we rewrite all this in terms of modified $I$ Bessel functions. We have the following well-known description of the $I$-Bessel function of order $\nu$ in terms of contour integrals (see \cite{Apostol}, p. 109)
\begin{equation}\label{eqBessel}
I_{\nu} (z) = \frac{ \left(\frac{1}{2} z\right)^\nu}{2 \pi i} \int_{c - i \infty}^{c + i \infty} t^{- \nu -1} e^{t + \frac{z^2}{4 t}} dt.
\end{equation}
Furthermore, one can express the $I$-Bessel functions whose order is half of an odd integer as an elementary function, e.g.,
$$
I_{\frac{1}{2}} (z) = \sqrt{\frac{2}{\pi z}} \sinh(z) = \frac{1}{\sqrt{2 \pi z}} ( e^z - e^{-z}).
$$

It is straightforward to show
\begin{align*}
\left| \int_{\frac{\pi}{12} + i \frac{\pi }{12} \sqrt{n^2 - 1}}^{\frac{\pi}{12} + i \infty} t^{- \frac{3}{2}} e^{ \frac{2 \pi^2 n}{12 t k^2} - \frac{\pi^2}{t 12^2 k^2} + t} dt\right|= O(n^{- \frac{1}{2}})
\end{align*}
and
$$
\left| \int_{\frac{\pi}{12} - i \infty}^{\frac{\pi}{12} - i \frac{\pi }{12} \sqrt{n^2 - 1}   } t^{- \frac{3}{2}} e^{ \frac{2 \pi^2 n}{12 t k^2} - \frac{\pi^2}{t 12^2 k^2} + t} dt\right| = O( n^{-\frac{1}{2}}). 
$$
Applying the trivial bound $| B_k(n) | \le C_0 k$, we have the following: 
\begin{align*}
T_2 &=   -\left(\frac{\pi}{12} \right)^{\frac{1}{2}} \frac{1}{\varphi(N)} \sum_{h,k} B_k(n) k^{- \frac{3}{2}} \left( \int_{\frac{\pi}{12} - i \infty}^{\frac{\pi}{12} + i \infty}  t^{-\frac{3}{2}} 
e^{\frac{2 \pi^2 n}{12 t k^2}- \frac{ \pi^2 }{t 12^2 k^2} + t}  dt + O(n^{-\frac{1}{2}}) \right)+ O(n^{2}) \\
&=  -\left( \frac{\pi}{12} \right)^{\frac{1}{2}} \frac{1}{\varphi(N)} \sum_{h,k} B_k(n) k^{- \frac{3}{2}} \int_{\frac{\pi}{12} - i \infty}^{\frac{\pi}{12} + i \infty}  t^{-\frac{3}{2}} 
e^{\frac{2 \pi^2 n}{12 t k^2}- \frac{ \pi^2 }{t 12^2 k^2} + t}  dt +  O(n^2).
\end{align*}

We take $z = 2 \sqrt{ \frac{\pi^2}{6k^2} \left( n - \frac{1}{24} \right)}$ in \eqref{eqBessel} and find

\begin{equation}\label{eqAsymp}
T_2 =-2 \pi i \frac{1}{\varphi (N) } \left( \frac{ \pi }{12} \right)^{\frac{1}{2}} \left( \frac{ \pi^2}{6} \left(n - \frac{1}{24} \right) \right)^{- \frac{1}{4}} \sum_{k=1}^{n} B_k(n) k^{-1} I_{\frac{1}{2}} \left( \frac{ \pi}{k} \sqrt{ \frac{2}{3} \left( n - \frac{1}{24} \right)} \right) + O(n^2).
\end{equation}

To compute the asymptotic, we need the $k=1$ term of this sum. We find with \eqref{eqZeta} and \Cref{thm:cuspbehavior} the following formula for $B_1(n) $,
$$
B_1(n) = a_0(0,1) = - \frac{1}{2 Ni} \sum_{\substack{ \psi \smod{N} \\ \psi(-1) = -1 }} \psi(r') \sum_{c = 0}^{N - 1} \psi(c) \cot \left( \frac{ \pi c}{N} \right).$$
We compute the error:
\begin{align*}
&2 \pi  \frac{1}{\varphi (N) } \left( \frac{ \pi }{12} \right)^{\frac{1}{2}} \left( \frac{ \pi^2}{6} \left(n - \frac{1}{24} \right) \right)^{- \frac{1}{4}} \sum_{k=2}^{n} \left|B_k(n) k^{-1} I_{\frac{1}{2}} \left( \frac{ \pi}{k} \sqrt{ \frac{2}{3} \left( n - \frac{1}{24} \right)} \right)\right| \\
&\le  \frac{\pi}{\varphi (N) }  \left(n - \frac{1}{24} \right)^{- \frac{1}{2}} C_0 \left( n^{\frac{5}{2}}   e^{\frac{ \pi}{2} \sqrt{ \frac{2}{3} \left( n - \frac{1}{24} \right)}} + O(1) \right)
\end{align*}

Thus we have:
\begin{equation}\label{eqErrorterm}
T_2 = \frac{1}{2 \sqrt{2} \varphi (N) N} \left( \sum_{\psi(-1) = -1} \psi( r ')  \sum_{c = 0}^{N - 1} \psi(c) \cot \left( \frac{ \pi c}{N} \right) \right) \frac{e^{ \left(\pi \sqrt{ \frac{2}{3} \left( n - \frac{1}{24} \right)} \right)}}{\sqrt{ \left( n - \frac{1}{24} \right)} }
+ O\left( n^2 e^{ \left( \frac{\pi}{2} \sqrt{ \frac{2}{3} \left( n - \frac{1}{24} \right)} \right)} \right). 
\end{equation}

This concludes the proof of \Cref{thm:main}. 

\subsection{Proof of Theorem \ref{thm:main2}}
We have already treated asymptotics for the difference of the number of parts in a partition of an integer $n$ which lie in two complementary congruence classes modulo a given number $N$. We did this with the help of the modularity of certain weight $1$ Eisenstein series. The ``limit case'', which still remains to be treated, is that, where we count parts in partitions of $n$ which are congruent to $0$ modulo $N$. Recall that we denote by $T_{0,N}(n)$ the number of these parts, then we have, as we have seen before, the identity of generating functions
\begin{equation}\label{T0t}
\sum\limits_{n=1}^\infty T_{0,N}(n)q^n=\sum\limits_{m=1}^\infty \frac{q^{Nm}}{(1-q^{Nm})^2}\prod\limits_{n\neq Nm}(1-q^n)^{-1}=\left(\sum\limits_{m=1}^\infty \sigma_0(m)q^{Nm}\right)\cdot\prod\limits_{n=1}^\infty(1-q^n)^{-1}.
\end{equation}
The sum 
\[S_0(\tau)=\sum\limits_{m=1}^\infty \sigma_0(m)q^{m}\]
on the right-hand side turns out to be essentially the period function of the Maa{\ss}-Eisenstein series $E(\tau;\tfrac 12)$. These period functions were introduced with a focus on those of Maa{\ss} cusp forms by Lewis and Zagier in \cite{LZ}, the non-cuspidal case is treated for example in \cite{CM}. A very convenient exposition for our purposes can be found in \cite{BC}. We require the following result (see Theorem 1 in \cite{BC} and Corollary 4.5 in \cite{RhoadesNgo}).
\begin{lemma}\label{periodasymp}
\begin{enumerate}
\item Let $M\in\N$. As $\sigma\rightarrow 0$ in the bounded cone $|\arg \sigma|<\tfrac \pi 2-\delta$ and $|\Im \sigma|\leq \pi$, one has
\[S_0\left(\frac{i\sigma}{2\pi}\right)=-\frac{\log \sigma}{\sigma}+\frac{\gamma_E}{\sigma}+\frac 14+\sum\limits_{n=1}^M\frac{B_{2n}^2}{(2n)!(2n)}\sigma^{2n-1}+O(\sigma^{2M+\frac 12}),\]
where $B_{2n}$ denotes the $2n$-th Bernoulli number.
\item For $\sigma\rightarrow 0$ in the cone $\tfrac \pi 2-\delta\leq |\arg \sigma|<\tfrac \pi 2$ and $|\Im\sigma|\leq \pi$ we have
\[S_0\left(\frac{i\sigma}{2\pi}\right)\ll_\delta t^{-\frac 32}.\]
\end{enumerate}
\end{lemma}
With this we are able to prove \Cref{thm:main2}.
\begin{proof}[Proof of \Cref{thm:main2}]
According to \eqref{T0t}, we can write 
\[T_{0,N}(n)=\frac{1}{2\pi i}\int_\calC \frac{L(q)\xi(q)}{q^{n+1}}dq,\]
where $\calC$ is a circle around $0$ with radius less than $1$, $L(q)=(2\pi)^{-\frac 12}q^{\frac{1}{24}}S_0(q^N)$ and $\xi(q)=\frac{(2\pi)^\frac 12}{\eta(q)}$. From \Cref{periodasymp} we see that $L(e^{-\sigma})$ satisfies the hypotheses \ref{hypo1} and \ref{hypo3}, while $\xi(e^{-\sigma})$ is well-known to satisfy hypotheses \ref{hypo2} and \ref{hypo4}. To be more precise, $L(e^{-\sigma})$ is a sum of the functions $$L_1(e^{-\sigma})=-\frac{e^{-\frac{\sigma}{24}}}{N(2\pi)^\frac 12}\frac{\log\sigma}{\sigma},$$
which is of logarithmic type near $1$, the relevant parameters being $B=1$ and $\alpha_0=-\tfrac{1}{N(2\pi)^{\frac 12}}$, and 
\[L_2(e^{-\sigma})=\frac{e^{-\frac{\sigma}{24}}}{N(2\pi)^\frac 12}\left(\frac{\gamma_E-\log N}{\sigma}+\frac N4+\sum\limits_{m= 1}^\infty\frac{B_{2m}^2}{(2m)!(2m)}N^{2m}\sigma^{2m-1}\right),\]
which is of polynomial type near $1$, the relevant parameters being again $B=1$ and $\alpha_0=\tfrac{\gamma_E-\log N}{N(2\pi)^\frac 12}$. The function $\xi(e^{-\sigma})$ satisfies hypotheses \ref{hypo2} and \ref{hypo4} with the parameters $\beta=\tfrac 12$, $c^2=\tfrac{\pi^2}{6}$, $\gamma=4\pi^2$ (see Theorem 4.1 in \cite{RhoadesNgo}). Plugging this into Propositions \ref{Wrightlog} and \ref{Wrightpoly} yields the result.
\end{proof}

\section*{Competing interests}
The authors declare that they have no competing interests in the present manuscript.

\end{document}